\documentclass{amsart}
\usepackage{amsthm,amssymb}
\usepackage{tikz}
\usetikzlibrary{arrows}
\usepackage{verbatim}
\usepackage{hyperref}

\newtheorem{thm}{Theorem}

\newtheorem{conj}[thm]{Conjecture}

\theoremstyle{definition}

\newtheorem{rem}[thm]{Remark}

\newcommand{\ddb}{\sqrt{-1}\partial\bar{\partial}}
\renewcommand{\phi}{\varphi}
\renewcommand{\[}{\begin{equation}}
\renewcommand{\]}{\end{equation}}

\begin{document}
\title{A remark on conical K\"ahler-Einstein metrics}
\author{G\'abor Sz\'ekelyhidi}
\address{Department of Mathematics, University of Notre Dame, Notre
  Dame, IN 46615}
\email{gszekely@nd.edu}

\begin{abstract}
We give some non-existence results for K\"ahler-Einstein metrics with
conical singularities along a divisor on Fano manifolds. In particular we show that the
maximal possible cone angle is in general smaller than the invariant
$R(M)$. We study this discrepancy from the point of view of log
K-stability.  
\end{abstract}
\maketitle

\section{Introduction}
Given a Fano manifold $M$ and a smooth anticanonical divisor $D\subset
M$, the existence of a K\"ahler-Einstein metric on $M$ with conical
singularities along $D$ has received considerable attention
recently. Interest in such metrics goes back to at least
McOwen~\cite{McO88} on Riemann surfaces, and Tian~\cite{Tian94} for
higher dimensions. The renewed interest has been sparked by a
proposal by Donaldson~\cite{Don09, Don11_1} to use such singular
metrics in a continuity method for finding smooth K\"ahler-Einstein
metrics, which has recently led to a solution of the problem of when
K\"ahler-Einstein metrics exist on Fano manifolds~\cite{CDS12}. 
There is by now a large body of work
on such conical
K\"ahler-Einstein metrics, see for
instance Mazzeo-Rubinstein~\cite{MR12}, Song-Wang~\cite{SW12},
Li-Sun~\cite{LS12}, and many others. 

In this paper we give some simple calculations implying non-existence
results. A K\"ahler-Einstein metric $\omega$ on $M$ with conical
singularities along a divisor $D\in |-K_M|$
satisfies the equation
\begin{equation}\label{eq:coneq}
 \mathrm{Ric}(\omega) = \beta\omega + (1-\beta)[D],
\end{equation}
where the cone angle is $2\pi\beta$ for some $\beta\in (0,1]$, and
$[D]$ denotes the current of integration along $D$. Let us
write
\[ R(M,D) = \sup\{ \beta > 0\,|\, \text{ there is a cone-singularity
  solution of \eqref{eq:coneq}}\}. \] 
Let $M_1$ and $M_2$ be the blowup of $\mathbf{P}^2$ in one or two
points respectively. 
\begin{thm}\label{thm:main}
On $M_1$, for any smooth $D\in |-K_{M_1}|$ we have $R(M_1,D) \leqslant
12/15$. On $M_2$, if $D\in |-K_{M_2}|$ passes through the
intersection of two $(-1)$-curves, then $R(M_2,D) \leqslant 7/9$. 
\end{thm}
Recall that for any Fano manifold $M$ one can define an invariant $R(M)\in(0,1]$
\[ R(M) = \sup \{ t\,|\, \exists \omega\in c_1(M)\,\text{ such that
}\mathrm{Ric}(\omega) > t\omega\}. \]
We computed in~\cite{GSz09} that $R(M_1) = 6/7$, and the invariant for
all toric Fano manifolds has been computed by Li~\cite{Li11} (see also
Tian~\cite{Tian92} for earlier results bounding $R(M)$). In
particular $R(M_2) = 21/25$. In \cite{GSz09} we proved that if
$\alpha\in c_1(M)$ is a K\"ahler form, then the equation
\[ \mathrm{Ric}(\omega) = \beta\omega + (1-\beta)\alpha \]
can be solved if and only if $\beta < R(M)$. In
relation to conical K\"ahler-Einstein metrics, i.e. when replacing
$\alpha$ by a current of integration along a smooth divisor,
Donaldson~\cite{Don11_1} conjectured the following.
\begin{conj}\label{conj} Suppose $D\in |-K_M|$ is smooth. 
 For all $0 < \beta < R(M)$ there exists a cone-singularity
  solution to \eqref{eq:coneq}, and there is no solution for $R(M) < \beta
  < 1$. In other words, $R(M,D)= R(M)$ for any smooth $D\in |-K_M|$. 
\end{conj}
Since $12/15 < 6/7 = R(M_1)$, and $7/9 < 21/25= R(M_2)$,
 our result gives counterexamples to this conjecture.

An important generalization of Equation~\eqref{eq:coneq} was studied
by Song-Wang~\cite{SW12}, where $D$ is allowed to be an element of the
linear system $|-mK_M|$ for some $m > 0$. In Section~\ref{sec:pluri}
we will give a non-existence result for conical K\"ahler-Einstein
metrics along such $D$, complementing the results of Song-Wang to some
extent. 

The proof of Theorem~\ref{thm:main} will be given in
Section~\ref{sec:proof}. It is based on a log K-stability calculation of
Li~\cite{Li11_1}, together with the result of Berman~\cite{Ber12}
which says that log K-stability is a necessary condition for the
existence of a conical K\"ahler-Einstein metric.
In Section~\ref{sec:stab} we will give a
discussion of the difference between $R(M)$ and $R(M,D)$
 from the point of view of
algebro-geometric stability conditions. 

\subsection*{Acknowledgements} I am grateful to Jian Song and Simon
Donaldson for helpful comments. In addition I would like to thank
Liviu Nicolaescu for suggesting the reference~\cite{SZ99}.

\section{Proof of Theorem~\ref{thm:main}}\label{sec:proof}
We will use the notion of log K-stability, which was introduced in
\cite{Don11_1} (see also \cite{GSzThesis} for a related notion for
asymptotically cuspidal metrics instead of conical ones). In
particular we will use the calculation in Li~\cite{Li11_1}, where
this stability notion is analyzed for toric manifolds. We quickly
recall his result. A toric Fano manifold $M$ can be viewed as a reflexive
lattice polytope $P$ in $\mathbf{R}^n$. For instance $M_1$, the blowup of
$\mathbf{P}^2$ in one point, 
corresponds to the convex hull of the points
$(0,-1),(-1,0),(-1,2),(2,-1)$ in $\mathbf{R}^2$, shown in
Figure~\ref{fig:M1}. 

The lattice points in $P$ correspond to sections of $K_M^{-1}$, giving
a decomposition of $H^0(M,K_M^{-1})$ into one-dimensional weight
spaces of the torus action. Let us write $\{s_1,\ldots,s_N\}$ for
these sections, corresponding to lattice points
$\{\alpha_1,\ldots,\alpha_N\}$. Given an anticanonical divisor $D$, we
can write
\[ D = \{ \sum_{i=1}^N a_is_i = 0\}, \]
for some coefficients $a_i$. Define $P_D\subset P$ to be the convex hull of
those weights $\alpha_i$, for which $a_i\not=0$. 

Let us choose $\lambda\in\mathbf{Z}^n$ giving the weights of a
one-parameter subgroup in $(\mathbf{C}^*)^n$. Note that $P$ naturally
lives in the dual of the Lie algebra of the torus, so here we are
identifying this $\mathbf{R}^n$ with its dual, using the Euclidean
inner product. This $\lambda$ defines a 
test-configuration for the pair $(M,D)$, which is simply a product
configuration on $M$, but degenerates $D$. Let us write
\[ W(\lambda) = \max_{p\in P_D}\langle p,\lambda\rangle,\]
and let $P_c\in P$ denote the barycenter of $P$. 
For any $\beta\in [0,1]$,
the Futaki invariant, denoted by $F(M,\beta D,\lambda)$ is
computed in Li~\cite{Li11_1} (see also Section~\ref{sec:stab} for more
details). The calculation there assumes that $D$
is generic, so that $a_i\not=0$ for all $i$ and so $P_D=P$, but the
same argument works if $P_D\not= P$. The result is 
\begin{thm}[Li~\cite{Li11_1}]\label{thm:Li}
\[ F(M,\beta D,\lambda) = -\Big[\beta\langle P_c,\lambda\rangle +
(1-\beta)W(\lambda)\Big]\mathrm{Vol}(P). \]
\end{thm}
The sign convention is such that logarithmic K-stability requires
\[ \label{eq:Fleq0}
F(M,\beta D,\lambda) < 0. \]
In particular Berman~\cite{Ber12} has shown that \eqref{eq:Fleq0} is
necessary for a conical KE metric to exist with angle $2\pi\beta$
along $D$. 

\begin{proof}[Proof of Theorem~\ref{thm:main}]
Let $D\subset M_1$ be a smooth anticanonical divisor. Suppose that $D$
intersects the exceptional divisor at the point $p$. We can choose a
torus action on $M_1$ for which $p$ is a fixed point. The toric
polytope $P$ can be chosen to be the convex hull of the points
$(0,-1),(-1,0),(-1,2),(2,-1)$, so the center of mass is given by
\[ P_c = \left(\frac{1}{12},\frac{1}{12}\right). \]
 Let us write $\{s_1,\ldots,s_N\}$ for
the sections of $K_{M_1}^{-1}$ giving eigenvectors of the torus
action, and let us assume that $s_N$ is the section corresponding to
the weight
$(-1,0)$.  We can assume that $p$ corresponds
to the vertex $(-1,0)$, meaning that the space of sections of
$K_{M_1}^{-1}$ which vanish at $p$ are spanned by the sections
$s_1,\ldots,s_{N-1}$. In Figure~\ref{fig:M1}, we have indicated the
lattice points corresponding to the
sections $s_1,\ldots, s_{N-1}$.  

\begin{figure}[h]
\begin{tikzpicture}[
  scale=0.7,
  >=stealth]
\draw[gray!30!white, very thin, fill] (0,-1) -- (-1,1) -- (-1,2) -- (2,-1) -- (0,-1);
\draw[gray, very thin] (-2.4,-2.4) grid (3.4,3.4);
\draw[thin, ->] (-2.5,0) -- (3.5,0);
\draw[thin, ->] (0,-2.5) -- (0,3.5);
\draw[thick] (0,-1) -- (-1,0) -- (-1,2) -- (2,-1) -- (0,-1);
\draw[thick, fill] (0,-1) circle (0.15cm);
\draw[thick, fill] (0,0) circle (0.15cm);
\draw[thick, fill] (0,1) circle (0.15cm);
\draw[thick, fill] (-1,1) circle (0.15cm);
\draw[thick, fill] (-1,2) circle (0.15cm);
\draw[thick, fill] (1,0) circle (0.15cm);
\draw[thick, fill] (1,-1) circle (0.15cm);
\draw[thick, fill] (2,-1) circle (0.15cm);
\node[above left] at (-1,0) {$p$};
\draw[thick, fill] (-1,0) circle (0.05cm);
\node[below left] at (-0.35, -0.35) {$Q$};
\draw[thick, fill] (-0.5, -0.5) circle (0.05cm);
\end{tikzpicture}
\caption{The polytope corresponding to $M_1$, with the sections
  vanishing at $p$ highlighted, and $P_D$ shaded.} \label{fig:M1}
\end{figure}
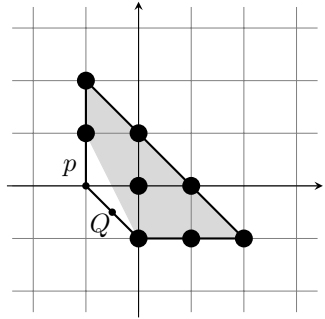

This implies that 
\[ D = \{ \sum_{i=1}^{N-1} a_is_i = 0\}, \]
for some coefficients $a_i$, and in particular
\[ P_D \subset \mathrm{conv}\{(0,-1),(-1,1),(-1,2),(2,-1)\}, \]
where ``conv'' denotes convex hull. 
Let us choose $\lambda = (-2,-1)$, and consider the test-configuration
corresponding to the one-parameter subgroup of $(\mathbf{C}^*)^2$
generated by $\lambda$. Using Theorem~\ref{thm:Li} we can compute
\[ F(M_1,\beta D,\lambda) = -\Big[\frac{-3}{12}\beta +
1-\beta\Big], \]
and $F(M_1,\beta D,\lambda) < 0$ implies $\beta < 12/15$. Theorem 4.2
of Berman~\cite{Ber12} implies that there is no conical metric
solution of \eqref{eq:coneq} for $\beta\geqslant 12/15$. 

For the manifold $M_2$ we can argue similarly. We have drawn the
corresponding polytope $P$ in Figure~\ref{fig:M2}. We can assume that we
chose our torus action in such a way, that the anticanonical divisor
$D$ meets two exceptional divisors at the point corresponding to the
vertex $p$. It follows that
$D$ is given as the zero set of a linear combination of the sections
corresponding to the lattice points $(-1,1),(0,-1),(0,0),(0,1),
(1,-1),(1,0)$ and $(1,1)$. The barycenter of $P$ is
\[ P_c = \left(\frac{2}{21}, \frac{2}{21}\right). \]

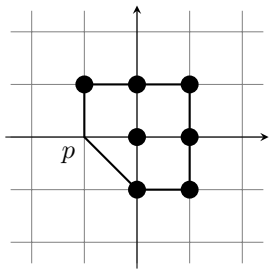
\begin{figure}[h]
\begin{tikzpicture}[
  scale=0.7,
  >=stealth]
\draw[gray, very thin] (-2.4,-2.4) grid (2.4,2.4);
\draw[thin, ->] (-2.5,0) -- (2.5,0);
\draw[thin, ->] (0,-2.5) -- (0,2.5);
\draw[thick] (0,-1) -- (-1,0) -- (-1,1) -- (1,1) -- (1,-1) -- (0,-1);
\draw[thick, fill] (0,-1) circle (0.15cm);
\draw[thick, fill] (0,0) circle (0.15cm);
\draw[thick, fill] (0,1) circle (0.15cm);
\draw[thick, fill] (-1,1) circle (0.15cm);
\draw[thick, fill] (1,1) circle (0.15cm);
\draw[thick, fill] (1,0) circle (0.15cm);
\draw[thick, fill] (1,-1) circle (0.15cm);
\node[below left] at (-1,0) {$p$};
\end{tikzpicture}
\caption{The polytope corresponding to $M_2$, with the sections
  vanishing at $p$ highlighted.} \label{fig:M2}
\end{figure}

Let us once again choose $\lambda=(-2,-1)$, and compute
\[ F(M_2, \beta D,\lambda) = -\Big[\frac{-6}{21}\beta +
(1-\beta) \Big],  \]
We find that $F(M_2, \beta D,\lambda) < 0$ implies $\beta <
7/9$. Once again, Berman's theorem~\cite{Ber12} implies that there is
no conical metric solution of \eqref{eq:coneq} for $\beta \geqslant
7/9$. 
\end{proof} 

\begin{rem}
Many other similar examples can be given. In general if $P_c$ is the
barycenter of the moment polytope $P$, let $Q$ be the intersection of the
ray from $P_c$ through the origin $O$, with the boundary of $P$. It is
shown by Li~\cite{Li11} that
\[ R(M) = \frac{|OQ|}{|P_cQ|}. \]
Using the formula in Theorem~\ref{thm:Li} is it easy to see that we
will get
\[ F(M, R(M)D, \lambda) < 0 \]
for a suitable $\lambda$ whenever $P_D$ does not contain the point
$Q$, as shown in Figure~\ref{fig:M1}.
\end{rem}

\section{Pluri-anticanonical divisors}\label{sec:pluri}
Instead of letting $D$ be an anticanonical divisor, we can allow $D$
to be a smooth divisor 
in the linear system $|-mK_M|$ for some $m > 1$. In this case
Song-Wang~\cite{SW12} have shown that for any $\beta\in(0, R(M))$
there exists an $m >0$ and a smooth divisor $D\in |-mK_M|$ 
so that there is a conical K\"ahler-Einstein metric $\omega$
satisfying the equation
\begin{equation}\label{eq:coneq2}
 \mathrm{Ric}(\omega) = \beta\omega + \frac{1-\beta}{m}[D]. 
\end{equation}
We give a related result in the converse direction.
\begin{thm}
On the manifold $M_1$, for any $m > 0$ there is a smooth divisor $D\in
|-mK_{M_1}|$ such that a cone-singularity solution of \eqref{eq:coneq2}
must satisfy  
\[ \beta < \frac{12m}{14m+1} < R(M_1).\]
Similarly on $M_2$ there is a smooth divisor $D\in |-mK_{M_2}|$ such
that a solution of \eqref{eq:coneq2} must satisfy
\[ \beta < \frac{21m}{25m+2} < R(M_2). \]
\end{thm}
\begin{proof}
We can use the same toric calculation as in the proof of
Theorem~\ref{thm:main}, using the polytope in Figure~\ref{fig:M1}. The
only difference is that sections of $K_{M_1}^{-m}$ correspond to
lattice points in $P\cap \frac{1}{m}\mathbf{Z}^2$. 

Let us write $s_0,\ldots, s_N$ for the corresponding sections, ordered
in such
a way that $s_0,\ldots,s_{m-1}$ correspond to the lattice points along the
edge joining $(-1,0)$ and $(0,-1)$, except for the point $(0,-1)$. In
other words these are the $m$ sections corresponding to the lattice
points
\[ (-1,0), \left(-\frac{m-1}{m},
  -\frac{1}{m}\right),\left(-\frac{m-2}{ m}, \frac{-2}{m}\right) 
\ldots, \left(-\frac{1}{m}, -\frac{m-1}{m}\right)\]
We will take $D$ to be of the form
\[ D = \{ \sum_{i=m}^N a_is_i = 0\},\]
for generic choice of $a_i$. This will be a smooth section by
Bertini's theorem, since the
base locus of the corresponding linear system consists of only the
point $p$, and we can check directly that the general element is
smooth at $p$. In fact to be smooth at $p$ we only need the
coefficient corresponding to the lattice point
$\left(-1,\frac{1}{m}\right)$ to be non-zero. The divisor $D$ will
meet the exceptional divisor with multiplicity $m$ at the point $p$. 

We now take $\lambda = (-m-1,-m)$. Again using Theorem~\ref{thm:Li}
(or rather a slight generalization which works for pluri-anticanonical
divisors), we obtain
\[ F(M_1, \beta D,\lambda) = -\Big[\frac{-2m-1}{12}\beta +
m(1-\beta)\Big].\]
The inequality $F(M_1, \beta D,\lambda) < 0$ implies
\[ \beta < \frac{12m}{14m+1}. \]
Note that for any $m$, we have $\frac{12m}{14m+1} < R(M_1)$, since
$R(M_1)=6/7$. It is also worth pointing out that for $m > 1$
the divisor $D$ we
use here is quite special, since a generic element in $|-mK_{M_1}|$
will meet the exceptional divisor in $m$ distinct points. 

The calculation for $M_2$ is completely analogous, the only difference
is that in that case $P_c = \left(\frac{2}{21},\frac{2}{21}\right)$ as
in the proof of Theorem~\ref{thm:main}. The divisor $D$ in this case
will meet the $(-1)$-curve which intersects the two exceptional
divisors, with multiplicity $m$ at the point $p$. 
\end{proof}

\section{Stability conditions}\label{sec:stab}
By definition $t < R(M)$ if and only if
there is a metric $\omega\in c_1(M)$, and a smooth 
positive form $\alpha\in c_1(M)$ such that
\begin{equation}\label{eq:Aubin}
 \mathrm{Ric}(\omega) = t\omega + (1-t)\alpha. 
\end{equation}
We showed in \cite{GSz09} that the solvability of
\eqref{eq:Aubin} for a given $t$ is independent of the choice of
$\alpha\in c_1(M)$. The reasoning behing Conjecture~\ref{conj} is the
natural expectation that the same holds if we allow $\alpha$ to be a
current supported on a divisor. We have seen that this is not the case
for the manifolds $M_1$ and $M_2$. 

To understand the counterexamples from the point of view of algebraic
geometry, we will compare log K-stability with an analogous notion of
stability where the current $[D]$ is replaced by a smooth form in
$c_1(M)$. We plan to flesh out these ideas in more detail in future work,
so for now we just give a brief sketch.

A test-configuration for $M$ is obtained by embedding $M\hookrightarrow
\mathbf{P}^{N_r}$ using the linear system $|-rK_M|$ for some $r > 0$,
and then acting on $\mathbf{P}^{N_r}$ by a $\mathbf{C}^*$-action $\lambda$. The
flat limit
\[ M_0 = \lim_{t\to 0} \lambda(t)\cdot M \]
is invariant under the action $\lambda$, and this can be used to
define (see Donaldson~\cite{Don02} for details) the Futaki invariant 
$\mathrm{Fut}(M,\lambda)$. 
Our sign convention, 
in order to match with Li~\cite{Li11_1}, is such that K-semistability means
$\mathrm{Fut}(M,\lambda) \leqslant 0$ for all such
test-configurations. 

In~\cite{Don09}, Donaldson outlined a modification of this, which is 
conjecturally equivalent
 to the existence of K\"ahler-Einstein metrics on $M$ with conical
singularities along a divisor $D\in |-mK_M|$ for some $m > 0$. Given a
test-configuration as above, we have $D\subset
M\subset\mathbf{P}^{N_r}$, and we can take the flat limit
\[ D_0 = \lim_{t\to 0}\lambda(t)\cdot D. \]
Suppose that $\lambda(t) = t^A$ for some
$A\in\sqrt{-1}\mathfrak{su}(N_r+1)$ with integer eigenvalues. For real
$t$, the one parameter group of automorphisms $\lambda(t)$ is induced
by the gradient flow of the function
\[ H_A = \frac{A_{ij}Z^i\overline{Z}^j}{|Z|^2}, \]
where the $Z^i$ are homogeneous coordinates on
$\mathbf{P}^{N_r+1}$.  It is well known that the function
\[ f(t) = \int_{\lambda(t)\cdot D} H_A\,\omega_{FS}^{n-1} \]
is increasing in $t$, where $n$ is the dimension of $M$, and
$\omega_{FS}$ is the Fubini-Study metric. One defines the Chow weight
to be 
\[ \mathrm{Ch}(D, \lambda) = \lim_{t\to 0} f(t). \]

The relevant modified Futaki invariant when looking for 
K\"ahler-Einstein metrics on $M$ with conical singularities along $D$, is
\[ \label{eq:FutD}
 \mathrm{Fut}(M,\beta D,\lambda) = \beta\mathrm{Fut}(M,\lambda) +
\frac{1-\beta}{m}\,\mathrm{Ch}(D,\lambda). \]
Here, as before, the parameter $\beta\in (0,1]$ determines the cone
angle. 

If we want to replace $D$ with a smooth positive form $\alpha\in
c_1(M)$, then it is natural to define an analogous Chow weight as
follows, as was also remarked on in Donaldson~\cite{Don09}. 
Let us write $\iota : M\hookrightarrow\mathbf{P}^{N_r+1}$ for
our initial embedding, and $\phi_t = \lambda(t)\circ\iota$. 
One can then check that the function
\[ f(t) = \int_M \alpha\wedge
\phi_t^*(H_A\omega_{FS}^{n-1}) \]
is monotonic in $t$, and we define
\[ \mathrm{Ch}(\alpha, \lambda) = \lim_{t\to 0} f(t). \]
Then in analogy with \eqref{eq:FutD} we define 
\[ \mathrm{Fut}(M, \beta\alpha, \lambda) = \beta\mathrm{Fut}(M,\lambda) +
(1-\beta)\,\mathrm{Ch}(\alpha, \lambda). \]

The main point that we want to make is the following. 
\begin{thm}
Suppose that $\alpha\in c_1(M)$ is a smooth positive form as above,
and $D\in |-K_M|$. Then we have 
\[ \lim_{t\to 0} \int_{M}
\alpha\wedge\phi_t^*(H_A\omega_{FS}^{n-1}) \leqslant \lim_{t\to 0}
\int_{\lambda(t)\cdot D} H_A\omega_{FS}^{n-1}. \]
In other words, we have 
\begin{equation}\label{eq:ineq}
 \mathrm{Fut}(M, \beta\alpha, \lambda) \leqslant \mathrm{Fut}(M, \beta
 D, \lambda) 
\end{equation}
for all $\beta\in [0,1]$, and all $\mathbf{C}^*$-actions $\lambda$. 
\end{thm}
\begin{proof}
First let us suppose that $\alpha$ is the pullback of a Fubini-Study
metric, i.e. $\alpha = \frac{1}{k}\Phi^*\omega_{FS}$ for some
embedding $\Phi:M\hookrightarrow \mathbf{P}^{N_k}$ using the linear
system $|-kK_M|$. In this case we can write $\alpha$ as an average of the
currents of integration $\frac{1}{k}[C]$ as the divisor $C$ varies over
$|-kK_M|$. This follows from Lemma 3.1 in
Shiffman-Zelditch~\cite{SZ99}. In fact the relevant measure $d\mu$ on the
linear system $|-kK_M|$ is induced by the inner product on
$H^0(K_M^{-k})$, for which the embedding $\Phi$ is given by
orthonormal sections.  

This implies that
\[ \int_M \alpha\wedge \phi_t^*(H_A\omega_{FS}^{n-1}) = \frac{1}{k}
\int_{C\in|-kK_M|} \left(\int_{\lambda(t)\cdot C}
  H_A\omega_{FS}^{n-1}\right)\,d\mu. \] 
For a fixed $C\in|-kK_M|$, the limit
\[ \lim_{t\to 0}\int_{\lambda(t)\cdot C} H_A\omega^{n-1} \]
is the Chow weight $\mathrm{Ch}(C,\lambda)$. For any integer $w$, let us
write $E_w\subset |-kK_M|$ for the set
\[ E_w = \{ C\in |-kK_M|\,\,;\, Ch(C,\lambda)\geqslant w \}. \]
This is a Zariski closed subset, since the weight can only jump up
under specialization. In fact under an embedding of $|-kK_M|$
into a projective space using the Chow line bundle, 
$E_w$ is the intersection with a linear subspace.
 It follows that if we let $w_{min}$ be the largest $w$ for which
$E_w = |-kK_M|$, then $E_{w_{min}+1}$ has measure zero in $|-kK_M|$. From
the monotone convergence theorem we obtain
\begin{equation}\label{eq:monotone} \begin{aligned}
\lim_{t\to 0} \int_M\alpha\wedge \phi_t^*(H_A\omega_{FS}^{
  n-1}) &= \frac{1}{k}\int_{C\in |-kK_M|}\left(\lim_{t\to 0}
\int_{\lambda(t)\cdot C} H_A\omega_{FS}^{n-1}\right)\,d\mu\\ &= \frac{1}{k}
\int_{C\in |-kK_M|\setminus E_{w_{min}+1}} w_{min}\,d\mu \\ &=
\frac{1}{k}w_{min}.
\end{aligned} \end{equation}
On the other hand, for a divisor $D\in |-K_M|$ we have $kD\in
|-kK_M|$, and so $kD\in E_w$ for some $w\geqslant w_{min}$. It follows
that
\[ \lim_{t\to 0} \int_{\lambda(t)\cdot D} H_A\omega_{FS}^{n-1} =
\frac{1}{k}\lim_{t\to 0} \int_{\lambda(t)\cdot kD} H_A\omega_{FS}^{n-1} =
\frac{1}{k}w \geqslant \frac{1}{k}w_{min}. \]
Comparing this with \eqref{eq:monotone} we obtain the result for such
$\alpha$. 

Now suppose that $\alpha\in c_1(M)$ is an arbitrary smooth positive
form. From the asymptotic expansion of the Bergman kernel (see
Tian~\cite{Tian90_1}, Ruan~\cite{Ru98}, Zelditch~\cite{Zel98}), we
know that we can approximate $\alpha$ with forms of the type
$\frac{1}{k}\Phi^*\omega_{FS}$. In
particular we can choose $\alpha_k\in c_1(M)$ for which our arguments
above apply, and 
\[ \alpha = \alpha_k + \ddb f_k, \]
where $|f_k| < \frac{1}{k}$. For any $t$ we have
\[ \int_M (\alpha - \alpha_k)\wedge \phi_t^*(H_A\omega_{FS}^{n-1}) =
\int_M f_k\, \phi_t^*(\ddb H_A\,\wedge\omega_{FS}^{n-1}). \]
For some constant $A$ we have
\[ -A\omega_{FS}^n < \ddb H_A\wedge \omega^{n-1} < A\omega_{FS}^n, \]
so 
\[ \left| \int_M f_k\,\phi_t^*(\ddb H_A\wedge\omega_{FS}^{n-1})\right|
< \frac{A}{k}\mathrm{Vol}(M). \]
It follows, using also \eqref{eq:monotone} that
\[\begin{aligned}
 \lim_{t\to 0} \int_M\alpha\wedge\phi_t^*(H_A\omega_{FS}^{n-1})
&= \lim_{t\to 0} \int_M\alpha_k\wedge\phi_t^*(H_A\omega_{FS}^{n-1})
+ O(1/k) \\ &= w_{min} + O(1/k). \end{aligned}\]
Since $k$ was arbitrary, we get
\[ \lim_{t\to 0} \int_M\alpha\wedge\phi_t^*(H_A\omega_{FS}^{n-1})
= w_{min},\]
and so the result follows for arbitrary smooth positive $\alpha \in
c_1(M)$. 
\end{proof}

\begin{rem}
It is clear from the proof that if 
$D$ is chosen to be in a special position, in
particular if it passes through more non-minimal critical points of
$H_A$ than a generic $D$ would, then one would expect strict inequality to
hold in \eqref{eq:ineq}.
This means that if we can find a cone-singularity solution of
\eqref{eq:Aubin}, with $\alpha=[D]$ for some divisor $D\in |-K_M|$,
then we expect to be able to solve the same equation with any smooth
positive form 
$\alpha\in c_1(M)$, at least if there are no holomorphic vector fields
on $M$. The converse, however, need
not be true if $D$ is in special position. This is exactly what happens
in the examples that we have for $M_1$ and $M_2$. In particular for
$M_1$, if $D$ is any smooth anticanonical divisor, then we can choose
a $\mathbf{C}^*$-action on $M_1$ for which $D$ is in special
position and gives a discrepancy between $R(M_1)$ and $R(M_1,D)$. 

It also follows from the proof that if we fix the
$\mathbf{C}^*$-action $\lambda$, then for a generic divisor $D$, we will
have equality in \eqref{eq:ineq}. A special case of this can be
observed in Theorem~\ref{thm:Li}, where for generic $D$ we have
$P_D=P$. Indeed in this case the formula matches up with the result we
obtained in \cite{GSz09} for the case of a smooth positive $\alpha\in
c_1(M)$, which was formulated in terms of the derivative of the
twisted Mabuchi functional. 
\end{rem}

It is interesting to speculate on what happens with the conical
K\"ahler-Einstein metrics on $M_1$, as $\beta\to 12/15$. Along the
test-configuration that we used in the proof of
Theorem~\ref{thm:main}, the divisor $D$ degenerates into a divisor
 $D_0$ given by the union of
a conic passing through the exceptional divisor, and a line
which is tangent to the conic. We expect that $M_1$ admits a
cone-singularity solution of 
\[ \mathrm{Ric}(\omega) = \beta\omega + (1-\beta)[D_0] \]
in a suitable sense with $\beta = 12/15$, to
which the conical KE metrics solving \eqref{eq:coneq} on 
$M_1\setminus D$ degenerate as $\beta\to
12/15$. Moreover, we expect that one can find K\"ahler-Ricci solitons
(or extremal metrics)
with conical singularities along $D_0$ in a suitable sense even for
$\beta > 12/15$. This would be a natural extension
of Donaldson's deformation result in \cite{Don11_1} to the case when
there exist vector fields preserving the divisor. Finally these
conical K\"ahler-Ricci solitons (or extremal metrics) should 
converge to the smooth K\"ahler-Ricci soliton
(or extremal metric), which is known to exist on $M_1$, as $\beta\to 1$.  

\bibliographystyle{amsplain}\bibliography{../mybib}

\end{document}